\documentclass{article}
\usepackage[a4paper, total={7in, 10in}]{geometry}
\usepackage[T1]{fontenc}
\usepackage{graphicx}
\usepackage{mathtools}
\usepackage{amsmath,amsthm,amssymb,mathrsfs,amstext, titlesec,enumitem, comment, graphicx, color, xcolor, stmaryrd,mathabx}
\usepackage{thmtools}
\usepackage{xpatch}
\usepackage{tikz-cd}
\usepackage{nameref}
\usepackage{hyperref}
\makeatletter
\def\Ddots{\mathinner{\mkern1mu\raise\p@
\vbox{\kern7\p@\hbox{.}}\mkern2mu
\raise4\p@\hbox{.}\mkern2mu\raise7\p@\hbox{.}\mkern1mu}}
\makeatother
\newtheorem{theorem}{Theorem}[section]
\newtheorem{corollary}[theorem]{Corollary}
\newtheorem{lemma}[theorem]{Lemma}
\newtheorem{question}[theorem]{Question}
\theoremstyle{definition}
\newtheorem{definition}[theorem]{Definition}

\begin{document}
\title{Monocromaticity of additive and multiplicative central sets and Goswami's theorem in large Integral Domains}

\date{}
	\author{Pintu Debnath
		\footnote{Department of Mathematics,
			Basirhat College,
			Basirhat-743412, North 24th parganas, West Bengal, India.\hfill\break
			{\tt pintumath1989@gmail.com}}
		\and
	Sourav Kanti Patra
	\footnote{Department of Mathematics,
		GITAM University Hyderabad Campus,
		Hyderabad, Telangana- 502329, India. \hfill\break
		{\tt spatra3@gitam.edu}	
		}
	}

\maketitle

\begin{abstract}
	In \cite{Fi} A. Fish proved that if $E_{1}$ and	$E_{2}$ are two subsets of $\mathbb{Z}$ of positive upper Banach density, then there exists $k\in\mathbb{Z}\setminus{0}$ such that $k\cdot\mathbb{Z}\subset\left(E_{1}-E_{1}\right)\cdot\left(E_{2}-E_{2}\right)$. In \cite{G}, S. Goswami proved  the same but a fundamental result on the set of prime numbers $\mathbb{P}$ in $\mathbb{N}$ and proved that for some  $k\in\mathbb{N}$,  $k\cdot\mathbb{N}\subset\left(\mathbb{P}-\mathbb{P}\right)\cdot\left(\mathbb{P}-\mathbb{P}\right)$. To do so, Goswami mainly proved that the product of an $IP^{\star}$-set with an $IP_{r}^{\star}$-set contains $k\mathbb{N}$. This result is very important and surprising to mathematicians who are aware of combinatorially rich sets. In this article,
	\begin{itemize}
		\item we extend Goswami's result to large Integral Domain, that behave like $\mathbb{N}$ in the  sense of some combinatorics.
		\item we prove that  for a combinatorially rich ($CR$-set) set, $A$, for some  $k\in\mathbb{N}$, $k\cdot\mathbb{N}\subset\left(A-A\right)\cdot\left(A-A\right)$.
		\item we provide a new proof that if we partition a large Integral Domain, $R$, into finitely many cells, then at least one cell is both additive and multiplicative central, and we prove the converse part, which is a new and unknown result.
	\end{itemize} 
\end{abstract}


compactifications of discrete semigroup.

\section{Introduction}

Let $\mathbb{P}$ be the set of prime numbers in $\mathbb{N}$. The famous conjecture of J. Maillet \cite{M5} is that \textbf{every even number can be written as the difference of two primes} i.e., $2\mathbb{N}\subset \mathbb{P}-\mathbb{P}$. In \cite{HuS} Huang and Sheng proved a surprising result in support of Millet's conjecture. To state this, we need some definitions. Let $\left(S,+\right)$ be a commutative semigroup and $A\subseteq S$.
	\begin{itemize}
		\item The set $A$ is $IP$-set if and only if  there exists a sequence $\left\{ x_{n}\right\} _{n=1}^{\infty}$ in $S$ such that $FS\left(\left\{ x_{n}\right\} _{n=1}^{\infty}\right)\subseteq A$. Where $$ FS\left(\left\{ x_{n}\right\} _{n=1}^{\infty}\right)=\left\{ \sum_{n\in F}x_{n}:F\in\mathcal{P}_{f}\left(\mathbb{N}\right)\right\}\,. $$
		
		\item The set $A$ is called $IP^{\star}$-set, when it intersects with all $IP$-sets.
		
		\item  Let $r\in \mathbb{N}$. The set $A$ is $IP_{r}$-set if and only if  there exists a sequence $\left\{ x_{n}\right\} _{n=1}^{r}$ in $S$ such that $FS\left(\left\{ x_{n}\right\} _{n=1}^{r}\right)\subseteq A$. Where $$ FS\left(\left\{ x_{n}\right\} _{n=1}^{r}\right)=\left\{ \sum_{n\in F}x_{n}:F\subseteq\{1,2,\ldots,r\}\right\}. $$
		
		\item Let $r\in \mathbb{N}$. The set $A$ is called $IP_{r}^{\star}$-set, when it intersects with all $IP_{r}$-sets.
		
		\item Let $r\in \mathbb{N}$. A set $A\subset\mathbb{N}$ is said to be a $\varDelta_{r}$-set if
		there exists a set $S\subset\mathbb{N}$ with $\vert S\vert=r$ such
		that $A=\left\{ s-t:s>t\text{ and }s,t\in S\right\} .$
		
		\item Let $r\in \mathbb{N}$. The set $A\subset\mathbb{N}$ is called $\varDelta_{r}^{\star}$-set, when it intersects with all $\varDelta_{r}$-set.
		
	\end{itemize}

In \cite{HuS} Huang and Sheng proved the following:
\begin{theorem}
	There exists $r\in\mathbb{N}$ such that $\mathbb{P}-\mathbb{P}$ is $\varDelta_{r}^{\star}$-set.
\end{theorem}
 Let $A$ be a subset of $\mathbb{N}$ with positive upper Banach density, then it is very well known result that $A-A$ is an $IP_{r}^{\star}$-set for some $r\in \mathbb{N}$. The set of prime numbers $\mathbb{P}$ has zero upper Banach density, but  a  simple consequence of the  above theorem is that $\mathbb{P}-\mathbb{P}$ is $IP_{r}^{\star}$-set in $\mathbb{N}$. Now $2\mathbb{N}$ is 	$IP_{2}^{\star}$-set. So we may claim that the result of Huang and Sheng is a partial answer to Maillet's conjecture on $\mathbb{P}-\mathbb{P}$. 

Let $A$ be a $CR$-set, introduced by Bergelson and Glasscock in \cite{BG}. The upper Banach density of $A$ may be zero. In spite of that,   in \textbf{Section 3}, we prove that for some $r\in \mathbb{N}$,  ${A-A}$ is an $IP_{r}^{\star}$-set for some $r\in \mathbb{N}$.

 In \cite{Fi} A. Fish proved that if $E_{1}$ and	$E_{2}$ are two subsets of $\mathbb{Z}$ of positive upper Banach density, then there exists $k\in\mathbb{Z}\setminus{0}$ such that $$k\cdot\mathbb{Z}\subset\left(E_{1}-E_{1}\right)\cdot\left(E_{2}-E_{2}\right).$$
 
Despite being  the zero upper Banach  density of $\mathbb{P}$, in \cite[Theorem 1.2]{G}  , S. Goswami proved  the following theorem:
\begin{theorem}{\cite[Theorem 1.12]{G}}
 There exists $k\in\mathbb{N}$ such that $$k\cdot\mathbb{N}\subset\left(\mathbb{P}-\mathbb{P}\right)\cdot\left(\mathbb{P}-\mathbb{P}\right).$$	
\end{theorem}
The main thing S. Goswami did to prove the above theorem is the following:
\begin{theorem}\label{Goswami's theorem}(Goswami's theorem \cite[Theorem 2.1]{G}):
	Let $r\in\mathbb{N}$ and let $A,B\subseteq\mathbb{N}$
	be $IP^{\star}$ set and $IP_{r}^{\star}$ sets respectively. Then
	there exists $k\in A$ such that $k\cdot\mathbb{N}\subseteq A\cdot B.$ 
\end{theorem} 
 
 In \textbf{Section 4}, We will extend the above theorem for large Lntegral Domain and as a consequence we will prove that for a $CR$-set  $A$ in $\mathbb{N}$, $$k\mathbb{N}\in (A-A)(A-A)$$ for some $k\in \mathbb{N}$.
 
 Now it is essential to  give a brief review of algebraic structure of the Stone-\v{C}ech
compactification of a semigroup $\left(S,\cdot\right)$, not necessarily commutative with the discrete topology to present what we want to do at remaining sections.

The set $\{\overline{A}:A\subset S\}$ is a basis for the closed sets
of $\beta S$. The operation `$\cdot$' on $S$ can be extended to
the Stone-\v{C}ech compactification $\beta S$ of $S$ so that $(\beta S,\cdot)$
is a compact right topological semigroup (meaning that for each    $p\in\beta$ S the function $\rho_{p}\left(q\right):\beta S\rightarrow\beta S$ defined by $\rho_{p}\left(q\right)=q\cdot p$ 
is continuous) with $S$ contained in its topological center (meaning
that for any $x\in S$, the function $\lambda_{x}:\beta S\rightarrow\beta S$
defined by $\lambda_{x}(q)=x\cdot q$ is continuous). This is a famous
Theorem due to Ellis that if $S$ is a compact right topological semigroup
then the set of idempotents $E\left(S\right)\neq\emptyset$. A nonempty
subset $I$ of a semigroup $T$ is called a $\textit{left ideal}$
of $S$ if $TI\subset I$, a $\textit{right ideal}$ if $IT\subset I$,
and a $\textit{two sided ideal}$ (or simply an $\textit{ideal}$)
if it is both a left and right ideal. A $\textit{minimal left ideal}$
is the left ideal that does not contain any proper left ideal. Similarly,
we can define $\textit{minimal right ideal}$ and $\textit{smallest ideal}$.

Any compact Hausdorff right topological semigroup $T$ has the smallest
two sided ideal

$$
\begin{aligned}
	K(T) & =  \bigcup\{L:L\text{ is a minimal left ideal of }T\}\\
	&=  \bigcup\{R:R\text{ is a minimal right ideal of }T\}.
\end{aligned}$$

Given a minimal left ideal $L$ and a minimal right ideal $R$, $L\cap R$
is a group, and in particular contains an idempotent. If $p$ and
$q$ are idempotents in $T$ we write $p\leq q$ if and only if $pq=qp=p$.
An idempotent is minimal with respect to this relation if and only
if it is a member of the smallest ideal $K(T)$ of $T$. Given $p,q\in\beta S$
and $A\subseteq S$, $A\in p\cdot q$ if and only if the set $\{x\in S:x^{-1}A\in q\}\in p$,
where $x^{-1}A=\{y\in S:x\cdot y\in A\}$. See \cite{HS} for
an elementary introduction to the algebra of $\beta S$ and for any
unfamiliar details.

  Let $A$ be a subset of $S$. Then $A$ is called central if and only if $A\in p$, for some minimal idempotent of $\beta S$ and $A$ is called central${}^{\star}$ if and only if  $A$ intersects with all central sets. From \cite[Corollary 5.21.1]{HS}, We know the following theorem:
\begin{theorem}
	Let $r\in \mathbb{N}$ and $\mathbb{N}=C_{1}\cup C_{2}\cup\ldots\cup C_{r}$. Then there exists $i\in \{1,2,\ldots,r\}$ such that $C_{i}$ is both additive and multiplicative central.
\end{theorem}
 
 In \textbf{Section 5}, we will prove that for an integral domain $R$, the above theorem is  true if and only if $R$ is a large Integral Domain.

\section{Large Integral Domain and $IP^{\star}$-sets}

\begin{definition}
Let $R$ be an integral domain. Then $R$ is called large Integral Domain,  if for any non-trivial ideal of $R$ is of finite index.
\end{definition}

The following theorem can easily explain why we consider Integral Domain  instead of arbitrary commutative ring with all non-trivial ideals are of finite indices. We think,  this is a very well known result in algebra but we don't have a specific reference, so we  present the proof.

\begin{theorem}
	Let $R$ be an infinite ring and  any non-trivial principal ideal of $R$ is of finite index. Then $R$ is an Integral domain i.e., large Integral Domain.
\end{theorem}
\begin{proof}
	Suppose $a,b\in R$, with $ab=0$. Let $I=aR$ and $J=bR$. As $R=\cup_{i=1}^{k} \left( b_{i} +bR\right)$ for some $k\in \mathbb{N}$. Now consider the set $ A=\left\{b_i:i=1,2,\ldots,k\right\}$.  Now we claim that $aR$ is a subset of $\{ay:y\in A\}$ and so $I$ is finite. As $R/I$ is also finite, it will then follows that $R$ is finite. \\
	 \textbf{Proof of claim}: Let $ax\in I$. There exists $y$ in $A$ such that $x-y\in J$. So there exists $z\in R$  such that $x=y+bz$. Hence, $ ax=ay $ and so $ax$ is in $\{ay:y\in A\}$.
\end{proof}

It is easy to prove that any non-trivial principal ideal of a large Integral Domain  is $IP^{\star}$-set. And it is interesting and surprising result that our well familiar ring $\mathbb{Z}[x]$ is not large Integral Domain. If we  consider the ideal $x\mathbb{Z}[x]$ which is not $IP^{\star}$-set in $\mathbb{Z}[x]$. The following question comes naturally to our mind:

\begin{question}
	What is the class of Integral Domains where all  nontrivial principal ideals are  $IP^{\star}$-sets?
\end{question}

Before answering this question, we will prove two consecutive lemmas.

\begin{lemma}\label{finite index ip*}
Let $G$ be a group and $H$ be a finite index subgroup of $G$ i.e.,
$\left[G:H\right]=r$ for some $r\in\mathbb{N}$. Then $H$ is $IP_{r+1}^{\star}$-set in $G$.
\end{lemma}

\begin{proof}
As $H$ be a finite index subgroup of $G$ with $\left[G:H\right]=r$ . Then $G=\cup_{i=1}^{r}g_{i}H$
for $\left\{ g_{1},g_{2},\ldots,g_{r}\right\} \subset H$. Let $\left\{ x_{n}\right\} _{n=1}^{r+1}$
be a finite  sequence in $G$. We have to establish that $FP\left(\left\{ x_{n}\right\}_{n=1}^{r+1} \right)\cap H\ne\emptyset$.
From ``pigeon hole principle'' , there exists $i\in\left\{ 1,2,\ldots,r\right\} $
such that, at least two elements of $\left\{ x_{1},x_{1}x_{2},\ldots,x_{1}x_{2}\cdots x_{r+1}\right\} $
belong to $g_{i}H$. So, let $\prod_{i=1}^{k_{1}}x_{i}\in g_{i}H$
and $\prod_{i=1}^{k_{1}+k}x_{i}\in g_{i}H$ for some $k_{1},k\in\left\{ 1,2,\ldots,r+1\right\} $.
Then, $\prod_{i=1}^{k_{1}}x_{i}=g_{i}h_{1}$ and $\prod_{i=1}^{k_{1}+k}x_{i}=g_{i}h_{2}$
for some $h_{1},h_{2}\in H$. Now, $\left(\prod_{i=1}^{k_{1}}x_{i}\right)^{-1}\left(\prod_{i=1}^{k_{1}+k}x_{i}\right)=h_{1}^{-1}h_{2}\in H$.
Hence, we get $\prod_{i=1}^{k_{1}}x_{i+k_{1}}\in H$, which is desired.
\end{proof}
\begin{lemma}\label{infinite index non ip*}
Let $G$ be a group and $H$ be an infinite index subgroup of $G$
i.e., $\left[G:H\right]=\infty$. Then $H$ is not an $IP^{\star}$-set
in $G$.
\end{lemma}

\begin{proof}
As $H$ is an infinite index subgroup of $G$ , we can find out a
sequence $\left\{ x_{n}\right\} _{n=1}^{\infty}$ in $G$, such that
$FP\left(\left\{ x_{n}\right\} \right)\cap H=\emptyset$, Where, \\

$$\begin{aligned}
	x_{1}  & \in  G\setminus H\\ 
	x_{2}  &\in  G\setminus H\cup x_{1}H\\
	x_{3}  &\in  G\setminus H\cup x_{1}H\cup x_{2}H\cup x_{1}x_{2}H\\
	\vdots
\end{aligned}$$

As, $\left[G:H\right]=\infty$,we can find out a sequence $\left\{ x_{n}\right\} _{n=1}^{\infty}$
in $G$, such that $FP\left(\left\{ x_{n}\right\} \right)\cap H=\emptyset$,
Where,

$$ x_{n}  \in  G\setminus H\cup\bigcup_{x\in FP\left(\left\{ x_{k}\right\} _{k=1}^{n-1}\right)}xH $$

for all $n\in\mathbb{N}$. Hence $H$ is not an $IP^{\star}$-set
in $G$.
\end{proof}
Combining Lemma \ref{finite index ip*}  and Lemma \ref{infinite index non ip*}  together we get the following theorem which is algebraically and combinatorially  very important and  nice result.
\begin{theorem}
Let $G$ be a group and $H$ be a subgroup of $G$. Then $H$ is an
$IP^{\star}$-set in $G$ if and only if $\left[G:H\right]<\infty$.
\end{theorem}

As a consequences of the above theorem, we get the following corollary,
which characterizes the family of all Integral Domain , where nontrivial
principal ideals are $IP^{\star}$-sets. And we get the answer of the question of this section. 
\begin{corollary}
Let $R$ be an Integral Domain. Every principal ideals of $R$ are
$IP^{\star}$-sets, if and only if $R$ is  large Integral Domain.
\end{corollary}

\section{Difference of $J$-sets and $CR$-sets}

\begin{definition}
	\textbf{(Følner sequence)} Let $S$ be a countable semigroup. A Følner sequence in $S$ is a sequence $\left\{F_{n}\right\}_{n=1}^{\infty}$ where each $F_{n}$ is a finite subset of $S$ and such that for any $s\in S$  the set  $sF_{n}:=\left\{sx:x\in F_{n}\right\}$ satisfies $$\lim_{n \rightarrow \infty} \dfrac{|gF_{n}\cap F_{n}|}{|F_{n}|}=1.$$
\end{definition}

\begin{definition}
    \textbf{(Upper Banach density)}	Let $S$ be a countable semigroup admititing a Følner sequence $\left\{F_{n}\right\}_{n=1}^{\infty}$.  Then for any subset $A\subset S$ we definie its upper density ( with respect to $\left\{F_{n}\right\}_{n=1}^{\infty}$) by  $$ \overline{d}_{\left\{ F_{n}\right\} }\left(A\right)=\limsup_{n \rightarrow \infty} \dfrac{|A\cap F_{n}|}{|F_{n}|}.$$ and upper Banach density is defined by $$ d^{\star}\left(A\right)=\sup\left\{ \overline{d}_{\left\{ F_{n}\right\} }\left(A\right):\left\{ F_{n}\right\} _{n=1}^{\infty} \text{ is a Følner sequence }\right\}. $$
\end{definition}
Here are some properties of density that we will use later in this article:
Let $S$ be countable semigroup  admitting a Følner sequence  and $A$ abd $B$ are subsets of $S$. Let $s\in S$, define $s^{-1}A=\left\{t\in S:st\in A\right\}$. Then
\begin{itemize}
	\item $d^{\star}\left(S\right)=1,$
	\item $d^{\star}\left(A\right)=d^{\star}\left(sA\right)=d^{\star}\left(s^{-1}A\right)\, \forall s\in S$ and
	\item $d^{\star}\left(A\cup B\right)=d^{\star}\left(A\right)+d^{\star}\left(B\right)$ if $A\cap B=\emptyset$.
\end{itemize}
It is well known result that if $A$ ia a positive upper Banach density set of $S$, then $A^{-1}A$ is an $IP^{\star}$-set.  For commutative semigroup $S$, if $d^{\star}\left(S\right)>0$ then $A-A$ is an $IP^{\star}$-set. A combinatorially rich subset of $\mathbb{N}$ is $\mathbb{P}$, set of prime numbers, which has a fundamental interest in all branch of number theory. The density of  $\mathbb{P}$ is zeo, but there is a surprising result is that $\mathbb{P}-\mathbb{P}$  is an $IP^{\star}$-set by \cite{G} and \cite{HuS}. Also, the  stronger result $\mathbb{P}-\mathbb{P}$  is an $IP_{r}^{\star}$-set for some $r\in\mathbb{N}$. Now, we concentrate on some combinatorially rich sets which have some properties like set of prime numbers $\mathbb{P}$. To do so,  we state the following definition.

\begin{definition}\label{J set}
Let $\left(S,+\right)$ be a commutative semigroup and let $A\subseteq S$.
A is $J$-set if and only if for every $F\in\mathcal{P}_{f}\left({}^{\mathbb{N}}{S}\right)$,
there exist $a\in S$ and $H\in\mathcal{P}_{f}\left(\mathbb{N}\right)$
such that for each $f\in F$, $$a+\sum_{n\in H}f(n)\in A.$$
\end{definition}

The set of prime numbers,  $\mathbb{P}$ contains arithmetic progression of arbitrary length and also $J$-sets in $\mathbb{N}$ and there exists a $J$-set in $\mathbb{N}$ with upper Banach density zero. From  this and above discussion, two question arise naturally.

\begin{question}\label{P and J} Let  $A$ be $J$ set in $\mathbb{N}$.	
	\begin{itemize}
		\item[(a)] Is the set of prime numbers $\mathbb{P}$ is a $J$-set in $\mathbb{N}$?
		\item[(b)] Is the set $A-A$ is an $IP_{r}^{\star}$-set for some $r\in\mathbb{N}$?
		\item[(c)] Does the sum of reciprocals of $A$ diverse, i.e.,$\sum_{n\in A}\frac{1}{n}=\infty$?
	\end{itemize}
\end{question}
 
 We strongly believe that the answer of  Question \ref{P and J} (a) and (c) are  very difficult and we present a partial answer of the Question \ref{P and J} (b).
 
\begin{theorem}\label{J-J}
	Let $A$ be a $J$-set of a commutative semigroup $S$. Then show
	that $A-A$ is an $IP^{\star}$-set in $S$.
\end{theorem}

\begin{proof}
	Let $\left\{ x_{n}\right\} _{n=1}^{\infty}$ be a sequence in $S$
	and also take two sequences $\left\{ y_{n}\right\} _{n=1}^{\infty}$
	and $\left\{ x_{n}+y_{n}\right\} _{n=1}^{\infty}$. Let two function
	$f,g:\mathbb{N}\rightarrow S$, defined by $f\left(n\right)=y_{n}$
	and $g\left(n\right)=x_{n}+y_{n}$ for all $n\in\mathbb{N}$. Then
	there exist $a\in S$ and $H\subset\mathcal{P}_{f}\left(\mathbb{N}\right)$,
	such that $a+\sum_{n\in H}x_{n}\in A$ and $a+\sum_{n\in H}x_{n}+y_{n}\in A$.
	So, we get $\sum_{n\in H}x_{n}\in A-A$. Hence $A-A$ is an $IP^{\star}$-set.
\end{proof}
 
In \cite{BG}, V. Bergelson and D. Glasscock introduced a new notion of
large sets in commutative semigroup $\left(S,+\right)$. They used
matrix notation. Given a $r\times k$ matrix $M$ we denote by $m_{i,j}$
the element in row $i$ and column $j$ of $M$.
\begin{definition}\label{CR set}
Let $\left(S,+\right)$ be a commutative semigroup and let $A\subseteq S$.
Then $A$ is combinatorially rich set (denoted $CR$-set) if and only
if for each $k\in\mathbb{N}$, there exists $r\in\mathbb{N}$ such
that whenever $M$ is an $r\times k$ matrix with entries from $S$,
there exist $a\in S$ and nonempty $H\subseteq\left\{ 1,2,\ldots,r\right\} $
such that for each $j\in\left\{ 1,2,\ldots,k\right\} $, $$a+\sum_{t\in H}m_{t,j}\in A.$$
\end{definition}

In \cite{HHST},  authors established an equivalent definition of $CR$-set,
which is compatible with the definition of $J$-set.
\begin{definition}\label{K-CR set}
Let $\left(S,+\right)$ be a commutative semigroup, let $k\in\mathbb{N}$,
and let $A\subseteq S$. Then A is a $k-CR$-set if and only if there
exists $r\in\mathbb{N}$ such that whenever $F\in\mathcal{P}_{f}\left({}^{\mathbb{N}}{S}\right)$,
with $|F|\le k$, there exist $a\in S$ and $H\in\mathcal{P}_{f}\left(\left\{ 1,2,\ldots,r\right\} \right)$
such that for all $f\in F$, $$a+\sum_{t\in H}f(t)\in A.$$
\end{definition}

Note that a set is a $CR$-set if and only if for each $k\in\mathbb{N}$,
it is a $k-CR$-set.

Using the technique of the Theorem \ref{J-J}, we get
\begin{theorem}
Let $A$ be a $2-CR$-set of a commutative semigroup $S$. Then show
that $A-A$ is an $IP_{r}^{\star}$-set in $S$ for some $r\in\mathbb{N}$.
\end{theorem}
\begin{corollary}\label{CR-CR}
Let $A$ be a $CR$-set of a commutative semigroup $S$. Then 
 $A-A$ is an $IP_{r}^{\star}$-set in $S$ for some $r\in\mathbb{N}$.	
\end{corollary}
\begin{proof}
	Follows from the fact that  $CR$-set is $2-CR$-set.
\end{proof}
 Using the above corollary with Goswami's Theorem \ref{Goswami's theorem}, we get the following corollary.
\begin{corollary}
	Let $A$ be a $CR$-set in  $\mathbb{N}$. Then for some $k\in \mathbb{N}$, $$ k\mathbb{N}\subseteq  \left(A-A\right)\left(A-A\right).$$ 
\end{corollary}

A question naturally arises as to whether the above results are true for non-commutative semigroup. For this curiosity, we first  need to state definition of $J$-set in arbitrary semigroups.

\begin{definition}\label{Non Com J set}
	Let $(S,\cdot)$ be an arbitrary semigroup. A set $A\subseteq S$ is said to be a  $J$-set if and only
	if for each $F\in\mathcal{P}_{f}\left({}^{\mathbb{N}}{S}\right)$  there exists $H\in {\mathcal{P}}_{f}(\mathbb{N})$ such that $H=\{t(1),t(2),\ldots,t(m)\}$ with $t(1)<t(2)<\ldots<t(m)$  for some $m\in \mathbb{N}$ such that
	for each $f\in F$, $$a(1)\cdot f\big(t(1)\big)\cdot a(2)\cdot f\big(t(2)\big)\cdot a(3)\cdots a(m)\cdot f\big(t(m)\big)\cdot a(m+1)\in A.$$
\end{definition}
Let $A$ be a $J$-set in a non-commutative semigroup. Intuitively, it is clear that for a non-commutative semigroup $S$, $A^{-1}A$ may not be an $IP^{\star}$-set.\\

\textbf{Example}: Let $F$ be a free semigroup with two generators $a$ and $b$. If we take $A=Fa$, then $A$ is a $J$ set but  $A^{-1}A$ is not an $IP^{\star}$-set in $F$, as  $A^{-1}A$ does not intersect with the $IP$-set, generated by $\{b^{n}\}_{n=1}^{\infty}$.


\section{Product of $IP_{r}^{\star}$-sets}

In \cite[Corollary 2.3]{G}, S. Goswami proved that, for the set of prime numbers $\mathbb{P}$ in $\mathbb{N}$, there exists $k\in \mathbb{N}$, such that $k\mathbb{N}\subseteq (\mathbb{P}-\mathbb{P})(\mathbb{P}-\mathbb{P})$ which is a very surprising and fundamental result in prime numbers. And this result coming from the following theorem:

\begin{theorem}\label{Goswami's theorem}(Goswami's theorem \cite[Theorem 2.1]{G}):
	Let $r\in\mathbb{N}$ and let $A,B\subseteq\mathbb{N}$
	be $IP^{\star}$ set and $IP_{r}^{\star}$ sets respectively. Then
	there exists $k\in A$ such that $k\cdot\mathbb{N}\subseteq A\cdot B.$ 
\end{theorem} 

Some questions naturally come up.
\begin{question}\label{question on product IP*}
	Let $A$ be an $IP^{\star}$-set and $B$ be an $IP_{r}^{\star}$-set in $\mathbb{N}$ for some $r\in\mathbb{N}$.
	\begin{itemize}
		\item[(a)] Does the  product of two $IP^{\star}$-set contain $k\mathbb{N}$ for some $k\in \mathbb{N}$?
		\item[(b)] Can we get some extra result in the product $BC$, where $C$ is an $IP_{s}^{\star}$-set in $\mathbb{N}$ for some $s\in\mathbb{N}$?
		\item[(c)] Can we extend the Theorem \ref{Goswami's theorem} for large Integral Domains ?
	\end{itemize}
\end{question}
In this section, we present an affirmative answers of the Question \ref{question on product IP*} (c).  Before discussing the upcoming results of this section, we want to clear that $\{0\}$ is not  $IP^{\star}$ set in $\left(R,+\right)$. If  $\{0\}$ is  $IP^{\star}$ set in $\left(R,+\right)$, then being large Integral Domain, $R$ is finite.
\begin{lemma}
Let $R$ be a large Integral Domain and $A$ be an   $IP^{\star}$-set
in $R$. Then for any $r\in R\setminus\left\{ 0\right\} $, $rA$
is an   $IP^{\star}$-set in $R$.
\end{lemma}

\begin{proof}
Let be $\left\{ x_{n}\right\} _{n=1}^{\infty}$ be a sequences in
$G$. As $\langle r\rangle$ is finite index, for some $k\in \mathbb{N}$,  $$G=\cup_{n=1}^{k}r_{i}+\langle r\rangle.$$
Atleast two elements of $\left\{ x_{1},x_{1}+x_{2},\ldots,x_{1}+x_{2}+\cdots+x_{k+1}\right\} $,
belong to $r_{i}+\langle r\rangle$, for some $i\in\left\{ 1,2,\ldots,k\right\} $.
So, there exists $H_{1}\subset\mathcal{P}_{f}\left(\mathbb{N}\right)$,
such that $$y_{1}=\sum x_{i}\in\langle r\rangle.$$ Now if we, consider
a new sequence $\left\{ x_{n}\right\} _{n=\max H_{1}+1}^{\infty}$,
we get $H_{2}\subset\mathcal{P}_{f}\left(\mathbb{N}\right)$, such
that $$y_{2}=\sum_{i\in H_{2}}x_{i}\in\langle r\rangle \text{ with } \max H_{1}<\min H_{2}.$$
In this process, we can construct a sequence $\left\{ y_{n}\right\} _{n=1}^{\infty}$
in $R$, where all $y_{n}$ are belong to $\langle r\rangle$ and $$y_{n}=\sum_{i\in H_{n}}x_{i}
\text{ with } \max H_{n}<\min H_{n+1}.$$ As $y_{n}\in\langle r\rangle$ for all
$n\in\mathbb{N}$, we can take $y_{n}=rz_{n}$ for some $z_{n}\in R$.
Now, as $A$ is an $IP^{\star}$-set, there exists a set $K\in\mathcal{P}_{f}\left(\mathbb{N}\right)$,
such that $\sum_{n\in K}z_{n}\in A$, which implies $$\sum_{n\in K}y_{n}=\sum_{n\in K}rz_{n}\in rA.$$
If we take $H=\cup_{n\in K}H_{n}$, then $\sum_{i\in H}x_{i}\in rA$. Hence $rA$ is an  $IP^{\star}$-set in $R$.
\end{proof}

\begin{theorem}
 Let $s\in\mathbb{N}$. Let $R$ be a large Integral Domain and $A$ be an $IP^{\star}$-set and $B$ be an $IP_{s}^{\star}$-set in $R$. Then there exists an $IP^{\star}$-set
$D$, such that $\langle r\rangle\setminus\left\{ 0\right\} \in AB$ for all $r\in D$.
\end{theorem}

\begin{proof}
As, $\left\{ 0\right\} $ is not an $IP^{\star}$-set in $R$, then
there exist $b_{i}\in R$ for $i=1,2,\ldots,s$, such that $FS\left(\left\{ b_{i}\right\} _{i=1}^{s}\right)\cap\left\{ 0\right\} =\emptyset$.
Let $$D=A\cap\bigcap_{y\in FS\left(\left\{ b_{i}\right\} _{i=1}^{s}\right)}yA,$$
which is an $IP^{\star}$-set in $R$. Let $x\in R\setminus\left\{ 0\right\} $,
$R$ being an Integral Domain, $FS\left(\left\{ xb_{i}\right\} _{i=1}^{s}\right)\cap\left\{ 0\right\} =\emptyset$.
As, $B$ be an $IP_{s}^{\star}$-set, $\sum_{i\in H(x)}xb_{i}\in B$
for some $H(x)\subseteq\left\{ 1,2,\ldots,s\right\} $. Let $r\in D$,
then $$r\in\left(\sum_{i\in H(x)}b_{i}\right)A\implies rx\in\left(\sum_{i\in H(x)}xb_{i}\right)A\subset AB,$$
which completes the proof.
\end{proof}
 
From the above theorem with Corollary \ref{CR-CR}, we get the following corollary:

\begin{corollary}
	Let $A$ be a $CR$-set in the large Integral Domain $R$. Then for some $r\in R\setminus {0}$, $$ rR\subseteq \left(A-A\right)\left(A-A\right).$$ 
\end{corollary}

\section{Multiplicative and additive central sets}

It is famous and well known result \cite[Corollary 5.21.1]{HS} that if we partition $\mathbb{N}$ into finitely many cells, then atleast one cell is both additive and multiplicative central. We will prove the same result for large Integral Domain, motivated by the result in additive and multiplicative $IP$-sets in \cite{BH1}. To do so,  we start this section by calling the  following lemma.

\begin{lemma}{\cite[Lemma 4.6]{BG}}	
	Let $\left(S,\cdot\right)$ and $\left(T,\cdot\right)$ be semigroups, $\phi:\left(S,\cdot\right)\rightarrow \left(T,\cdot\right)$ be a homomorphism, $A\subseteq S$. If $A$ is central in $S$ and $\phi \left(S\right)$ is piecewise syndetic in $T$, then $\phi \left(A\right)$ is central in $T$.
\end{lemma}

For a large Integral Domain, $\left(R,+,\cdot\right)$, $\phi:\left(R,+\right)\rightarrow \left(R,+\right)$ be a homomorphism defined by $\phi\left(x\right)=rx$. Here $\phi \left(R\right)$ is $IP^{\star}$-set for $r\in R\setminus \left\{0\right\}$. So from the above lemma it is clear that, $\phi$ preserve additive central set in $R$.

Another important point is that $\left\{0\right\}$ is not additive  $J$ set, hence not additive central  set in $R$.

\begin{lemma}
	Let $A$ be an additive central${}^{\star}$-set in  the large Integral Domain $R$. Then $r^{-1}A$ is also  an additive central${}^{\star}$-set in  $R$ for any $r\in R\setminus \{0\}$.
\end{lemma}

\begin{proof}
	To prove that $r^{-1}A$ is an additive central${}^{\star}$-set in $R$, it is sufficient to show that for any additive  central set $B$ in $R$, we have $B\cap r^{-1}A\neq \emptyset$. Since $B$ is an  additive central set in the large Integral Domain $R$, $rB$ is also an additive central set in $R$, so $rB\cap A\neq \emptyset$. Choose $s\in rB\cap A$  and $t\in B$ such that $s=rt$. For being Integral Domain, $rt\in rB$ implies $t\in B$. Hence $t\in B\cap r^{-1}A\neq \emptyset$.
\end{proof}

\begin{lemma}\label{adi central implies Multiplica thick}
 Let $R$ be a large Integral Domain.	Let $A$ be an additive central${}^{\star}$-set in $R$. Then $A$ is multiplicative thick set in  $R$ and as a consequence $A$ is multiplicative central in $R$.
\end{lemma}

\begin{proof}
	Let $k\in \mathbb{N}$. Let $F=\{r_{1},r_{2},\ldots,r_{k}\}\subset R\setminus \{0\}$. Then $$r_{1}^{-1}A/\cap r_{2}^{-1}A/\cap \ldots \cap r_{k}^{-1}A\neq \emptyset.$$ So, there exists $a\in A$ such that $Fa\subset A$.
\end{proof}

\begin{lemma}
Let $R$ be a large Integral Domain.	Let $A$ be an additive central${}^{\star}$-set in $R$. If we partition $A$ into finitely many cell i.e., $A=C_{1}\cup C_{2}\cup \ldots \cup C_{k}$ , then one cell is  multiplicative central in $R$.
\end{lemma}

\begin{proof}
	Follows from the Lemma \ref{adi central implies Multiplica thick} with the fact that thick set implies central set.
\end{proof}

\begin{theorem}\label{LiD mul ad central}
Let $R$ be a large Integral Domain. If we partition, $R$ into finitely many cells, then atleast one cell is both additive and multiplicative central set.
\end{theorem}

\begin{proof}
	Let $R=C_{1}\cup C_{2}\cup \ldots \cup C_{k}$. For simplicity, let only  $C_{1}, C_{2}, \ldots, C_{m}$ are all  additive central set. So, $$A=C_{1}\cup C_{2}\cup \ldots \cup C_{m}$$ is additive central${}^{\star}$-set in $R$. So atlest one $C_{i}$ for some $i\in \{{1,2,\ldots,m}\}$ is multiplicative central. 
\end{proof}

There are many Integral Domains known to us that are not large Integral Domains like $\mathbb{Z}[x]$, $\mathbb{F}[x]$ where $\mathbb{F}$ is any infinite field. A natural question arises as to whether the Theorem \ref{LiD mul ad central}  is true for all these  domains. We present a negative answer by the following theorem:

\begin{theorem}\label{non Lid }
Let $R$ be an Integral Domain, which is not large Integral Domain. Then there exists a partition of $R$, so that no cell is both additive and multiplicative central.
\end{theorem}

\begin{proof}
As $R$ is not a large Integral Domain, there exists $\alpha\in R$ such that $\left[R:\alpha R\right]=\infty$.  
Let  $$R=\alpha R\cup(R\setminus\alpha R).$$ We claim that  $\alpha R$ is not  additive
central and $R\setminus\alpha R$ is not multiplicative central.\\
\textbf{proof of the first claim}: If $\alpha R$ is central set, then $\alpha R$ is also $J$-set. As $\alpha R$ is  a  additive subgroup of $R$, we have $\alpha R$ is an additive $IP^{\star}$-set in $R$, which contradicts the fact that $\left[R:\alpha R\right]=\infty$.\\
\textbf{proof of the second claim}:  As the multiplicative upper Banach density is dilation invariant, multiplicative upper Banach density of $\alpha R$ is $1$.  So,  $R\setminus\alpha R$ is of zero multiplicative upper Banach density and $R\setminus\alpha R$ is not multiplicative central.
\end{proof}

Combining Theorem \ref{non Lid } and Theorem \ref{LiD mul ad central}, we obtain all class of  Integral Domains such that a partition of  a particular type of Integral Domain  has at least one cell both additive and  multiplicative central. Obviously we get the following theorem.

\begin{theorem}
	Let $R$ be an Integral Domain. Then at least  one cell of any  partition of $R$  is both additive and multiplicative central if and only if $R$ is large Integral Domain. 
\end{theorem}

\bibliographystyle{plain}

\end{document}